\documentclass[11pt]{article}

\usepackage{amsmath,amssymb,amsthm}
\usepackage{enumitem}
\usepackage[normalem]{ulem}
\usepackage[a4paper,margin=1in]{geometry}
\usepackage{xcolor}
\usepackage{subcaption}
\usepackage{graphicx}
\usepackage{authblk}

\theoremstyle{plain} 
\newtheorem{theorem}{Theorem}
\newtheorem{proposition}{Proposition}
\newtheorem{ob}{Observation}
\newtheorem{corollary}{Corollary}
\newtheorem{lemma}{Lemma}

\theoremstyle{definition}
\newtheorem{example}{Example}

\newcommand{\abs}[1]{\left|#1\right|}
\newcommand{\floor}[1]{\left\lfloor #1\right\rfloor}
\newcommand{\ceil}[1]{\left\lceil #1\right\rceil}

\usepackage{tikz}
\usepackage{pgfplots}
\pgfplotsset{compat=1.18}
\usepgfplotslibrary{fillbetween}
\usetikzlibrary{backgrounds}

\tikzset{
  tree edge/.style={
    draw=black,
    line width=0.6pt
  },
  vertex/.style={
    circle,
    fill=black,
    minimum size=4.5pt,
    inner sep=0pt
  },
  vertex label/.style={
    font=\scriptsize,
    fill=white,
    inner sep=0.8pt
  },
  resistant edge/.style={
    draw=black,
    line width=0.9pt,
    line cap=round,
    line join=round
  },
  proper resistant edge/.style={
    resistant edge,
    densely dotted
  },
  full resistant edge/.style={
    resistant edge,
    solid
  }
}

\newcommand{\rotatedresistantedge}[1]{%
  \begin{scope}[rotate=#1]
    \path[
      proper resistant edge,
      rounded corners=7pt
    ]
      (90:2.75) -- (210:2.75) -- (330:0.82) -- cycle;
  \end{scope}%
}

\tikzset{
  extremal tree/.style={
    x=1cm,
    y=1cm
  },
  tree vertex/.style={
    circle,
    draw=black,
    fill=white,
    line width=0.6pt,
    minimum size=7pt,
    inner sep=0pt
  },
  tree label/.style={
    font=\small,
    fill=white,
    inner sep=0.8pt
  },
  leaf label/.style={
    font=\small,
    align=center
  },
  construction edge/.style={
    draw=black,
    line width=0.6pt
  }
}

\theoremstyle{definition}
\newtheorem{definition}{Definition}

\newtheorem{remark}{Remark}
\setlist[enumerate,1]{label=(\arabic*)}

\title{Minimum Contagious Sets in Trees with Degree-Proportional Thresholds}
\author[1]{Jeong-Ok Choi\thanks{Corresponding author.
E-mail addresses: \texttt{jchoi351@gist.ac.kr} (J.-O. Choi),
\texttt{cmcdev@uchicago.edu} (C. McDevitt).}}
\author[2]{Collin McDevitt}

\affil[1]{Department of Mathematical Sciences,
Gwangju Institute of Science and Technology,
Gwangju 61005, South Korea}

\affil[2]{Data Science Institute,
University of Chicago,
Chicago, IL 60637, United States}

\date{}

\date{}

\begin{document}
\maketitle

\begin{abstract}
  We investigate how large the minimum size of a contagious set can be among trees of a fixed order. 
  Given a graph $G$ and $\rho\in[0,1]$, an inactive vertex $u$ becomes  active when at least $\ceil{\rho d_G(u)}$ of its neighbors are active. An initial set of active vertices is contagious if the process eventually activates every vertex. We obtain exact values for the minimum size of a contagious set in perfect $d$-ary trees. We then determine the exact maximum of this minimum size over all trees of order $n$. We establish this result by bounding the number of pairwise disjoint resistant sets, each of which must intersect every contagious set, and constructing trees that attain the resulting bounds. Finally, by bounding the change in the minimum size of a contagious set under a leaf exchange, we show that, for $0<\rho\leq1$ and $n\geq2$, every integer between $1$ and this maximum occurs as the minimum size of a contagious set in a tree of order $n$.
\end{abstract}


\medskip
\section{Introduction}
Threshold-based diffusion processes on graphs provide a natural framework for modeling the adoption and spread of behaviors, technologies, products, opinions, and information. In this paper, we consider such processes on finite simple graphs. Starting from a set of initially active vertices, an inactive vertex becomes active once sufficiently many of its neighbors are active, and remains active thereafter. A set of initially active vertices of a graph is called a \emph{contagious set} if every vertex eventually becomes active through this process. 

This basic threshold process has been studied under several names. In particular, contagious sets are also referred to as \emph{irreversible dynamic monopolies} or \emph{perfect target sets} in the literature. Many variants of the process have also been considered under different threshold functions. Examples include constant \cite{reichman2012}, general vertex-dependent \cite{zaker2012}, and degree-proportional threshold functions \cite{gentner2017}. 

More specifically, for a graph $G$, we focus on the degree-proportional threshold function $\tau_\rho: V(G) \longrightarrow \mathbb{N}_0$ defined by  \[
\tau_\rho(v)=\lceil \rho d_G(v)\rceil,
\qquad \rho\in[0,1].
\]
A contagious set with respect to $\tau_\rho$ will be called a
$\rho$-contagious set, and we write $h_\rho(G)$ for the minimum
size of a $\rho$-contagious set in $G$. This degree-proportional
threshold also arises naturally in game-theoretic models of contagion
based on coordination games; see, for example,
\cite{morris2000,yu2019}.

A fundamental question in the study of such processes is to determine or bound the minimum size of a contagious set. This question has been studied under various threshold models \cite{peleg1998,dreyer2009,chen2009,ackerman2010,
reichman2012,zaker2012} and under degree-proportional threshold functions \cite{chang2011,garbe2018}. In particular, Gentner and Rautenbach \cite{gentner2017} proved that for $\rho\in(0,1]$, every tree $T$ of order at least $1/\rho$ satisfies
\[
h_\rho(T)\le \rho |V(T)|.
\]

Since $h_\rho(T)$ depends on the structure of $T$, its value is not determined by the order of the tree and $\rho$ alone. This motivates the study of exact values for particular families of trees, such as perfect $d$-ary trees, as well as the range of possible values among all trees of a given order.

In this paper, we study $h_\rho$ on trees.
We first obtain an explicit formula for $h_\rho(T_{d,k})$ valid
for every $\rho\in[0,1]$, where $T_{d,k}$ denotes
the perfect $d$-ary tree with $k$ levels. 

We then consider the extremal problem of determining the maximum possible value of $h_\rho(T)$ among all $n$-vertex trees $T$. For $n\ge 2$, define
\[
M_\rho(n)
=
\max\{h_\rho(T): T \text{ is a tree on } n
\text{ vertices}\}.
\]
We determine $M_\rho(n)$ exactly for every
$\rho\in(0,1]$ and $n\ge 2$.

We also investigate how $h_\rho$ behaves under edge additions. We prove that for any graph adding a single edge changes $h_\rho$ by at most one. Moreover, we show that for every tree of order at least three, there exists an edge whose addition does not decrease $h_\rho$. 

Finally, using the behavior of $h_\rho$ under leaf exchanges, we determine the complete spectrum of $h_\rho$ on $n$-vertex trees. More precisely, we show that
for every $\rho\in(0,1]$ and $n\ge 2$,
\[
\{h_\rho(T): T \text{ is a tree on } n
\text{ vertices}\}
=
\{1,2,\ldots,M_\rho(n)\}.
\]

The paper is organized as follows.
In Section~2, we introduce the activation process
and establish some basic properties of contagious sets.
In Section~3, we determine $h_\rho$ for perfect
$d$-ary trees and examine its connection with
minimum vertex covers.
In Section~4, we characterize contagious sets
in terms of resistant sets and use a packing
argument for resistant subtrees to determine
the extremal function $M_\rho(n)$.
 Section~5 investigates the effect of edge additions on contagious sets and $h_\rho$, both for general graphs and for trees.
Finally, in Section~6, we use leaf exchanges to determine the complete spectrum of $h_\rho$ on $n$-vertex trees. 

\bigskip
\section{Preliminaries} 
Throughout this paper, all graphs are finite and simple. 
For a graph $G$, let $V(G)$ and $E(G)$ denote its vertex and edge sets, respectively. 
We write $ \abs{G}$ for the order of $G$. We use $N_G(v)$ and
$d_G(v)$ to denote the neighborhood and degree, respectively, of a vertex $v$
in $G$. When the graph $G$ is clear from context, we may use the notation $d(v)$ and $N(v)$ for the degree and neighborhood of a vertex $v$ in $G$.

Let $\tau \colon V(G) \to \mathbb{R}_{\geq 0}$ be a threshold function. We allow threshold functions to take the value zero. Let
$S \subseteq V(G)$ be a possibly empty set of initially \textit{active}
vertices. Define the activation process recursively by $S_0^G = S$ and
\[
S_{t+1}^G
=
S_t^G \cup
\bigl\{
v \in V(G) :
|N_G(v) \cap S_t^G| \geq \tau(v)
\bigr\},
\qquad t \geq 0.
\]

Note that $S_t^G\subseteq S_{t+1}^G$ for every $t\ge0$;
that is, once a vertex becomes active, it remains active.
When the underlying graph is clear from the context, we write
$S_t$ instead of $S_t^G$.

If there exists some $t \geq 0$ such that $S_t^G = V(G)$, then $S$ is called a
$\tau$-\textit{contagious set} of $G$
\cite{garbe2018}. A natural question is to determine the minimum size of a $\tau$-contagious set.
Accordingly, define
\[
h_{\tau}(G)
=
\min
\bigl\{
|S| :
S \text{ is a } \tau\text{-contagious set of } G
\bigr\}.
\]

We study a specific threshold function called the
\textit{degree-proportional} threshold function
\cite{garbe2018}.
For $\rho\in[0,1]$, define $\tau_\rho\colon V(G)\to\mathbb{N}_0$ by
$\tau_\rho(v)=\ceil{\rho d_G(v)}$ for all $v\in V(G)$. When the threshold
function under consideration is the degree-proportional function $\tau_\rho$,
we write $\rho$ in place of $\tau_\rho$ and use the term $\rho$-contagious set
instead of $\tau_\rho$-contagious set. When the value of $\rho$ is clear from the context we use the term \textit{contagious set} instead of $\rho$-contagious set. 

Define
\[
M_\rho(n)=\max\{h_\rho(T):T\text{ is a tree of order }n\}.
\]

It was shown in
\cite[Theorem~2]{gentner2017} that
\[
M_\rho(n)\leq\rho n
\]
for all $\rho\in(0,1]$ and $n\geq1/\rho$. In our work, we find the exact value of $M_{\rho}(n)$ across $\rho\in[0,1]$ and $n\geq 1$.

\begin{remark}
We briefly recall the coordination-game interpretation of the degree-proportional threshold; see, for example, \cite{morris2000, yu2019}. Consider a two-strategy coordination game with strategies $A$ and $B$.
If a vertex chooses $A$, then it receives payoff $1-\rho$ from each
neighbor choosing $A$ and payoff $0$ from each neighbor choosing $B$.
Similarly, if the vertex chooses $B$, then it receives payoff $0$ from
each neighbor choosing $A$ and payoff $\rho$ from each neighbor choosing
$B$. The total payoff of a vertex is the sum of the payoffs obtained from its neighbors.

Suppose that ties are resolved in favor of $A$.
If a vertex $u$ has $k$ neighbors choosing $A$ and $m$ neighbors
choosing $B$, then $u$ chooses $A$ if and only if
\[
(1-\rho)k \ge \rho m.
\]
Since $k+m=d_G(u)$, this is equivalent to
\[
k\ge \rho d_G(u),
\]
or, equivalently,
\[
k\ge \lceil \rho d_G(u)\rceil=\tau_\rho(u).
\]
Thus, if vertices choosing $A$ are regarded as active, then an inactive vertex $u$ becomes active exactly when the number $k$ of its active neighbors satisfies $k \ge \tau_\rho(u)$. 
\end{remark}

\medskip
We conclude this section with a basic observation about connected components. 

\begin{ob}\label{ob:lower}
Let $G$ be a finite graph, and let $G_1,\ldots,G_k$ be the
nontrivial connected components of $G$. If $\rho\in(0,1]$, then
\[
h_\rho(G)=\sum_{i=1}^k h_\rho(G_i).
\]
In particular,
\[
h_\rho(G)\ge k.
\]
\end{ob}

Indeed, the activation processes on distinct components are independent, since
\[ N_G(v) = N_{G_i}(v) \] for every $v \in V(G_i)$. 
Thus, the activation of a vertex in $G_i$ depends only on the vertices of $G_i$. 
Hence, a set $S\subseteq V(G)$ is $\rho$-contagious
if and only if $S\cap V(G_i)$ is $\rho$-contagious in $G_i$ for every
$i=1,\ldots,k$.

Isolated vertices make no contribution to
$h_\rho(G)$, since for an isolated vertex $v$,
\[
\tau_\rho(v)=\left\lceil \rho d_G(v)\right\rceil=0.
\]
Hence, an isolated vertex need not be initially active.

\section{Exact values of $h_\rho$ for perfect $d$-ary trees }

Note first that if $S$ is a vertex cover of $G$, then $|N(x) \cap S| = |N(x)| = d(x) \ge \lceil \rho d(x) \rceil$ for every $x \notin S$. So, $S_1^G = V(G)$, and hence $S$ is a $\rho$-contagious set. The converse also holds when the maximum degree of $G$ is bounded
relative to $\rho$, as shown in the following proposition.  

\begin{proposition}\label{prop:vcover}
Let $G$ be a graph with maximum degree at most $d+1$. If $\rho \in (\frac{d}{d+1}, 1]$, then every $\rho$-contagious set is a vertex cover of $G$. Hence, $h_{\rho}(G) = \beta(G)$, where $\beta(G)$ is the size of a minimum vertex cover of $G$. 
\end{proposition}
\begin{proof}
For $\rho \in (\frac{d}{d+1}, 1]$ let $S$ be a $\rho$-contagious set. If $S$ is not a vertex cover, then there exist adjacent vertices $x$ and $y$ such that $S \cap \{ x, y \} = \emptyset$. Without loss of generality, we may assume that $x$ is activated no later than $y$. Suppose that $t_0$ is the step at which $x$ is activated. Then \[ |N(x)\cap S_{t_0-1}|\ge \lceil \rho d(x)\rceil. \]
Since $y \in N_G(x)$ and $y\notin S_{t_0-1}$, we have \[ |N(x)\cap S_{t_0-1}|\le d(x)-1. \] On the other hand, since $d(x)\le d+1$,
\[
 \frac{d(x)-1}{d(x)} = 1-\frac{1}{d(x)} \le 1-\frac{1}{d+1} = \frac{d}{d+1}<\rho.
\]
Hence,
\[
\rho d(x)>d(x)-1.
\] Therefore, \[ d(x)-1 \ge |N(x)\cap S_{t_0-1}| \ge \lceil \rho d(x)\rceil > d(x)-1, \] a contradiction. Thus, $S$ is a vertex cover of $G$. The equality $h_\rho(G)=\beta(G)$ follows from the observation preceding the proposition.
\end{proof} 

Let $T_{d, k}$ be a perfect $d$-ary tree with $k$ levels,
in which the root is at level $1$, every non-leaf vertex has
exactly $d$ children, and all leaves are at level $k$.

\begin{proposition}\label{prop:beta}
For $d,k \ge 2$, 
\[
\beta(T_{d, k}) = 
\begin{cases}
\dfrac{d^k-1}{d^2-1}, & k \text{ even},\\[6pt]
\dfrac{d(d^{k-1}-1)}{d^2-1}, & k \text{ odd}.
\end{cases}
\]

\end{proposition}

\begin{proof}
For $1 \le i \le k$, let $L_i$ denote the set of vertices at level $i$.

For a graph $G$, if there is a matching $M_0$ and a vertex cover $Q_0$ such that $|M_0| = |Q_0|$, then $Q_0$ is a minimum vertex cover of $G$, since $|Q| \ge |M|$ for any matching $M$ and any vertex cover $Q$.

If $k$ is even, then we consider
\[
Q = L_1 \cup L_3 \cup \cdots \cup L_{2t+1} \cup \cdots \cup L_{k-1}.
\]
Also, for each $v \in Q$, choose a child of $v$ and let $e_v$ be the edge connecting them. Then
\[
M = \{e_v \colon v \in Q\}
\]
is a matching. Note that
\[
|M| = |Q| = 1+d^2+\cdots+d^{2i}+\cdots+d^{k-2}
= \dfrac{d^k-1}{d^2-1}.
\]
Hence, $Q$ is a minimum vertex cover.
\smallskip
Similarly, if $k$ is odd, we consider
\[
Q = L_2 \cup L_4 \cup \cdots \cup L_{2t} \cup \cdots \cup L_{k-1}.
\]
For each $v \in Q$, choose a child of $v$ and let $e_v$ be the edge connecting them. Then
\[
M = \{e_v \colon v \in Q\}
\]
is a matching. Note that
\[
|M| = |Q| = d+d^3+\cdots+d^{2i-1}+\cdots+d^{k-2}
= \dfrac{d(d^{k-1}-1)}{d^2-1}.
\]
Hence, $Q$ is a minimum vertex cover.

\end{proof}
\medskip

\begin{proposition}\label{prop:leafless}
Let $\rho\in[0,1]$, and let $G$ be a graph with no component isomorphic to $K_2$. Then there exists a minimum $\rho$-contagious set of $G$ containing no leaves. 
\end{proposition}
\begin{proof}
If $\rho = 0$, the result is immediate. Suppose that $\rho > 0$, and let $S$ be a minimum $\rho$-contagious set of $G$. Suppose that $S$ contains a leaf $v$, and let $u$ be its unique neighbor. Let
\[
S'=(S\cup\{u\})\setminus\{v\}.
\]

If $u \in S$, then $S' = S - \{ v \}$. Since $N(v) \cap S' = \{ u \}$ and $|N(v) \cap S'| \ge \lceil \rho d(v) \rceil = 1$, $v \in S'_1$. Hence, $S = S'_0 \cup \{v\} \subseteq S_1'$, and $S'$ is a $\rho$-contagious set, contradicting the minimality of $S$. So $u \notin S$. Now we have
\[ S \subseteq S \cup \{ u \} = S' \cup \{ v \} \subseteq S_1'.\]
Thus, $S'$ is $\rho$-contagious and $|S'| = |S|$. If $S'$ contains a leaf, say $v'$, let $u'$ be its unique neighbor.  Repeating the previous argument with \[ S''=(S'\cup\{u'\})\setminus\{v'\}, \] we obtain another minimum $\rho$-contagious set. Since $G$ has no component isomorphic to $K_2$, the vertex $u'$ is not a leaf. Hence $S''$ contains fewer leaves than $S'$. Repeating this process, we eventually obtain a minimum $\rho$-contagious set containing no leaves.

\end{proof}

Observe that, by the above proposition, to determine the minimum size of $\rho$-contagious sets it is sufficient to consider $\rho$-contagious sets containing no leaves.

\begin{theorem}
Let $d\ge2$ and $k\ge2$. Then
\[
h_\rho(T_{d,k})=
\begin{cases}
0, & \rho=0,\\[2mm]
1, & 0<\rho\le \dfrac{1}{d+1},\\[3mm]
d^{k-2}, & \dfrac{1}{d+1}<\rho\le \dfrac{d}{d+1},\\[3mm]
\beta(T_{d,k}), & \dfrac{d}{d+1}<\rho\le 1.
\end{cases}
\]

\end{theorem}

\begin{proof}
If $\rho = 0$, it is immediate. 
If $0<\rho\le \frac{1}{d+1}$, then for every vertex $v$,
\[
0<\rho d_{T_{d,k}}(v)\le 1,
\]
since $d_{T_{d,k}}(v)\le d+1$. Hence,
$\tau_\rho(v)=1$. Thus, a vertex becomes active as soon as at least
one of its neighbors is active. Therefore, $L_1$ is a
$\rho$-contagious set, as $L_2,L_3,\ldots,L_k$ are activated
successively. Since $T_{d,k}$ is nontrivial and connected, $h_\rho(T_{d,k})\ge1$ by Observation~\ref{ob:lower}. Hence,
\[
h_\rho(T_{d,k})=1.
\]
\smallskip

For $\frac{1}{d+1} < \rho \le \frac{d}{d+1}$, if $k=2$, then $T_{d,2}$ is a star, and the root itself forms a
$\rho$-contagious set. Hence,
\[
h_\rho(T_{d,2})=1=d^{k-2}.
\]
Thus, we may assume that $k\ge3$. Let $S$ be a minimum $\rho$-contagious set containing no leaves, whose existence is guaranteed by the previous proposition. Suppose that there is a vertex $u\in L_{k-1}$ such that $u\notin S$. Let $t_0$ be the smallest index such that $u\in S_{t_0}$. Note that none of the $d$ children of $u$ belongs to $S_{t_0-1}$, since they are leaves and cannot be activated before their unique neighbor $u$ is activated. Hence,
\[ 1 \ge |N(u) \cap S_{t_0-1}| \ge \lceil \rho d(u) \rceil = \lceil \rho (d+1) \rceil > 1\]
since $\rho > \frac{1}{d+1}$, a contradiction. Therefore, $L_{k-1} \subseteq S$, and hence $|S| \ge |L_{k-1}|$.

On the other hand, $L_{k-1}$ is a $\rho$-contagious set. Indeed, every
vertex in $L_{k-2}$ has all of its $d$ children in $L_{k-1}$ active, and
\[
\left\lceil \rho(d+1)\right\rceil\le d.
\]
Thus, every vertex in $L_{k-2}$ becomes active, and the same argument successively
activates $L_{k-3},\ldots,L_1$. Also, every leaf in $L_k$ is activated
by its parent in $L_{k-1}$. Hence, $L_{k-1}$ is a minimum
$\rho$-contagious set.

\smallskip
Finally, suppose that
\[
\frac{d}{d+1}<\rho\le 1.
\]
By Proposition~\ref{prop:vcover}, every $\rho$-contagious set of
$T_{d,k}$ is a vertex cover. Hence,
\[
h_\rho(T_{d,k})=\beta(T_{d,k}).
\]
The explicit value of $\beta(T_{d,k})$ is given in
Proposition~\ref{prop:beta}.

\end{proof}

\medskip
\begin{remark}
The known bound in \cite{gentner2017} gives
\[
h_\rho(T)\le \rho \abs{V(T)}
\]
for trees of sufficiently large order. For perfect $d$-ary trees,
the exact value obtained above sharpens this general estimate
considerably. In particular, for $d,k \ge2$ and
$\rho>\frac{d}{d+1}$,
\[
\frac{h_\rho(T_{d,k})}{\abs{V(T_{d,k})}}
\le \frac{1}{d+1}
<\rho.
\]

\end{remark}

\section{ Exact values of $M_{\rho}(n)$}

The main result of this section is the exact value of $M_{\rho}(n)$ for $\rho \in(0,1]$ and $n\geq 2$. Since $M_{\rho}(n)=0$ if and only if $\rho=0$ or $n=1$, we restrict our analysis to the nondegenerate cases.

\begin{theorem}\label{thm:main}
Let $0<\rho\le 1$ and $n\ge 2$.

If $\rho>\frac12$, then
\[
M_\rho(n)=\floor{\frac n2}.
\]

If $0<\rho\le\frac12$, then
\[
M_\rho(n)=
\begin{cases}
1,
&2\le n<2\left(\floor{\frac1\rho}+1\right),\\[1mm]
2,
&2\left(\floor{\frac1\rho}+1\right)
 \le n\le
 3\left(\floor{\frac1\rho}+1\right),\\[1mm]
\displaystyle
\floor{\frac{n-1}{\floor{1/\rho}+1}},
&n\ge 3\left(\floor{\frac1\rho}+1\right)+1.
\end{cases}
\]
\end{theorem}
We show the upper bounds in Lemma~\ref{lem:upper-bounds} and then construct trees attaining these bounds in Section~\ref{extremal-trees}.

\subsection{Resistant sets and contagious sets}

\begin{definition}[\cite{zaker2012,centeno2011}]
A nonempty set $F\subseteq V(G)$ is a $\tau$-\textit{resistant set} if
\[
|N_G(v)\setminus F|<\tau(v)
\]
for every $v\in F$. When $\tau$ is clear from the context,
we simply refer to $F$ as a \textit{resistant set}.
\end{definition}

\begin{lemma}\label{resistant graph prop}
A set $S\subseteq V(G)$ is a $\tau$-contagious set if and only if it intersects every
$\tau$-\textit{resistant set}.
\end{lemma}

The forward implication was established in \cite[Proposition~1(iii)]{centeno2011}, while both implications were established
under equivalent conditions in \cite[Proposition~2]{zaker2012}.
For clarity, we prove both implications under our assumptions.

\begin{proof}
Suppose that $S$ is disjoint from a resistant set $F$.
We will show by induction that $S_t^G\cap F=\emptyset$ for all $t\geq 0$.
The base case follows from our assumption. Suppose that
$F\cap S_t^G=\emptyset$. For any vertex $v\in F$, we have
\[
|N_G(v)\cap S_t^G|
\leq |N_G(v)\setminus F|
<\tau(v).
\]
Therefore, $v\notin S_{t+1}^G$, and we conclude that $S$ is not a $\tau$-contagious set.

Conversely, suppose that the process starting from $S$ terminates before every
vertex becomes active. Let $F$ be the set of vertices that remain inactive.
Then $F\cap S=\emptyset$. Since no
vertex of $F$ can be activated after the process has terminated, for
every $v\in F$ we have
\[
|N_G(v)\setminus F|<\tau(v).
\]
Thus, $F$ is a $\tau$-resistant set disjoint from $S$.
\end{proof}

\begin{lemma}\label{minimum-intersect-set}
Let $G$ be a graph, and let $\tau$ be a threshold function. Then
\[
h_{\tau}(G)
=
\min\{|U| : U \text{ intersects every connected }
\tau\text{-resistant set of }G\},
\]
where connected means that the induced subgraph is connected.
\end{lemma}

\begin{proof}
Let $F\subseteq V(G)$ be a resistant set, and let $F_1,\ldots,F_\ell$ be the vertex sets of the connected components of $G[F]$. For every $v\in F_i$, since $F_i$ is a connected
component of $G[F]$, the vertex $v$ has no neighbor in $F\setminus F_i$. Hence,
\[
N_G(v)\setminus F_i=N_G(v)\setminus F.
\]
Since $F$ is $\tau$-resistant,
\[
|N_G(v)\setminus F_i|
=
|N_G(v)\setminus F|
<
\tau(v).
\]
Thus, each $F_i$ is itself a connected $\tau$-resistant set. Therefore, a set $S\subseteq V(G)$ intersects every resistant set if and only if it intersects every connected resistant set. By Lemma~\ref{resistant graph prop}, this is equivalent to $S$ being a
$\tau$-contagious set. Hence, the minimum is taken over exactly the same family of sets as in the definition of $h_{\tau}(G)$.
\end{proof}

\begin{definition}
Let $\mathcal H=(V,E)$ be a hypergraph with vertex set $V$ and hyperedge set
$E\subseteq \mathcal{P}(V)\setminus\{\emptyset\}$, where $\mathcal{P}(V)$ is the power set of $V$. A \textit{transversal} is a set of vertices
$U\subseteq V$ such that, for every hyperedge $e\in E$, we have
$U\cap e\neq\emptyset$. The \textit{transversal number} is
\[
c(\mathcal H)=\min\{|U|:U\text{ is a transversal of }\mathcal H\}.
\]

\end{definition}

\begin{definition} Let $\mathcal H=(V,E)$ be a hypergraph. A \textit{matching} is a set of
hyperedges $M\subseteq E$ such that $e_i\cap e_j=\emptyset$ for all distinct
$e_i,e_j\in M$. The \textit{matching number} is
\[
\nu(\mathcal H)=\max\{|M|:M\text{ is a matching of }\mathcal H\}.
\]

\end{definition}

A hypergraph $\mathcal H$ is said to have the K\H{o}nig property if
$c(\mathcal H)=\nu(\mathcal H)$. Given a tree $T=(V,\mathcal E)$, we can
create an associated hypergraph $\mathcal H=(V,E)$ whose hyperedges are the
vertex sets of a family of subtrees of $T$. This associated hypergraph has
the K\H{o}nig property by  \cite[Proposition~2.8]{alain2013}. An example of such an associated hypergraph is illustrated in
Figure~\ref{fig:hyper-graph-k_1,3}.

\begin{corollary}\label{cor:min-max}
Let $T$ be a tree and $\tau:V(T)\to \mathbb N_0$ any threshold function.
Then 
    \[
    h_{\tau}(T)
=
\max\bigl\{|M|:M\text{ is a set of pairwise disjoint connected }
\tau\text{-resistant sets of }T\bigr\}.    
    \]
\end{corollary}
\begin{proof}
    Define the hypergraph $\mathcal H_{\tau}(T)$ to have
vertex set $V(T)$ and hyperedge set consisting of all connected
resistant sets. By Lemma~\ref{minimum-intersect-set}, we have
\[
h_{\tau}(T)=c(\mathcal H_{\tau}(T)).
\]
A matching in $\mathcal{H}_{\tau}(T)$ is a family of pairwise disjoint connected resistant sets of $T$. Since every hyperedge induces a subtree of $T$, the hypergraph $\mathcal{H}_{\tau}(T)$ has the K\H{o}nig property by \cite[Proposition~2.8]{alain2013}. Therefore, 
\[
h_{\tau}(T)= c(\mathcal{H}_{\tau}(T))= \nu(\mathcal{H}_{\tau}(T)).
\]

\end{proof}

\begin{example}
Let $T=K_{1,3}$ and $\rho=1/2$. Label the vertices by
$\{c,l_1,l_2,l_3\}$, where $c$ is the label for the center and
$l_1,l_2,$ and $l_3$ are the labels for the leaves. No singleton is a
$1/2$-resistant set, and any $1/2$-resistant set must be
connected, i.e., it must contain $c$. The set
$\{c,l_i\}$ is not a $1/2$-resistant set for any leaf $l_i$.
It is easy to check that $\{c,l_1,l_2\}$, $\{c,l_1,l_3\}$,
$\{c,l_2,l_3\}$, and $\{c,l_1,l_2,l_3\}$ are all the
resistant sets. The associated hypergraph is then illustrated in
Figure~\ref{fig:hyper-graph-k_1,3}. 

 \begin{figure}[ht]
        \centering
        \begin{tikzpicture}[font=\small]
  \coordinate (c)  at (0,0);
  \coordinate (l1) at (90:2.25);
  \coordinate (l2) at (210:2.25);
  \coordinate (l3) at (330:2.25);

  \begin{scope}[on background layer]
    \path[
      full resistant edge,
      rounded corners=10pt
    ]
      (90:3.08) -- (210:3.08) -- (330:3.08) -- cycle;
    \rotatedresistantedge{0}
    \rotatedresistantedge{120}
    \rotatedresistantedge{240}
  \end{scope}

  \draw[tree edge] (c) -- (l1);
  \draw[tree edge] (c) -- (l2);
  \draw[tree edge] (c) -- (l3);

  \node[vertex] at (c)  {};
  \node[vertex] at (l1) {};
  \node[vertex] at (l2) {};
  \node[vertex] at (l3) {};

  \node[vertex label,below=3pt] at (c)  {$c$};
  \node[vertex label,below=3pt] at (l1) {$l_1$};
  \node[vertex label,below=3pt] at (l2) {$l_2$};
  \node[vertex label,below=3pt] at (l3) {$l_3$};
\end{tikzpicture}
\caption{The hypergraph $\mathcal{H}_{1/2}(K_{1,3})$}
\label{fig:hyper-graph-k_1,3}
\end{figure}
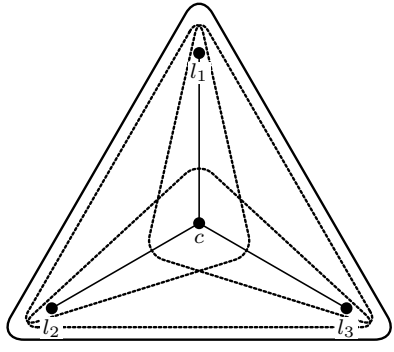

\end{example}

More generally, for $n\geq 1$ and $\rho \in (0,1]$, the center $c$ of $K_{1,n}$ is in every $\rho$-resistant set. Thus, by Lemma~\ref{resistant graph prop}, $h_{\rho}(K_{1,n})=1$.

\subsection{Bounds from resistant sets}
For a tree $T$ and a set of vertices $F\subseteq V(T)$, let $\delta_T(F)$ denote the set of edges with one endpoint in $F$ and the other in $V(T)\setminus F$.

\begin{lemma}\label{lem:lower-bounds-of-resistance}
Let $T$ be a tree, and let $F\subset V(T)$ be a proper connected $\rho$-resistant set.
\begin{enumerate}
\item If $\rho>\frac12$, then $\abs{F}\ge 2$.
\item If $0<\rho\le\frac12$, then
\[
\abs{F}\ge \floor{\frac1\rho}+1.
\]
If equality holds, then $\abs{\delta_T(F)}=1$.
\end{enumerate}

\end{lemma}
\begin{proof}
    For each $v\in F$, the resistant set condition and
$\rho\le 1/\floor{1/\rho}$ imply
\[
\floor{\frac1\rho}\abs{N_T(v)\setminus F}
\le d_T(v)-1.
\]
Therefore,
\begin{equation}\label{eq:local-resistant}
\left(\floor{\frac1\rho}-1\right)
\abs{N_T(v)\setminus F}
\le \abs{N_T(v)\cap F}-1.
\end{equation}

It is easy to see that $\sum_{v\in F}\abs{N_{T}(v)\setminus F}=\abs{\delta_{T}(F)}$. Moreover, since $F$ is connected and $T$ is a tree, the induced
subgraph $T[F]$ is a tree with $\abs F-1$ edges. Hence, by the
degree-sum formula,
\[
\sum_{v\in F}\abs{N_T(v)\cap F}
=
2\abs{E(T[F])}
=
2(\abs F-1).
\]
Thus, summing both sides of (\ref{eq:local-resistant}), we get 
\begin{equation}\label{eq:summed F}
  \left( \floor{\frac{1}{\rho}}-1\right) \abs{\delta_{T}(F)} \leq \abs F-2
\end{equation}

If $\rho>1/2$, then (\ref{eq:summed F}) gives $\abs{F}\geq 2$. Now assume $0<\rho \leq 1/2$. Since $F$ is proper and $T$ is connected, $|\delta_{T}(F)|\geq 1$. Then (\ref{eq:summed F}) gives 
$
\abs{F}\ge \floor{1/\rho}+1
$.
If 
$
\abs{F}= \floor{1/\rho}+1
$,
then (\ref{eq:summed F}) gives $\abs{\delta_{T}(F)}\leq 1$, so $\abs{\delta_{T}(F)}= 1$.
\end{proof}
\begin{proposition}\label{prop:min-order}
Let $T$ be a tree with $\abs{T}=n\geq 2$ and $m=h_{\rho}(T)$.
\begin{enumerate}
\item If $\rho>\frac12$, then $n\ge 2m$.
\item If $0<\rho\le\frac12$ and $m\ge3$, then
\[
n\ge m\left(\floor{\frac1\rho}+1\right)+1.
\]
\item If $0<\rho\le\frac12$ and $m=2$, then
\[
n\ge 2\left(\floor{\frac1\rho}+1\right).
\]
\end{enumerate}

\end{proposition}
\begin{proof}

By Corollary~\ref{cor:min-max}, $T$ contains $m$ pairwise disjoint connected resistant sets
\[
F_1,\dots , F_m.
\]

If $m=1$, then the result follows immediately from $n\geq 2$. Otherwise, $m\geq 2$, which means that each $F_i$ is proper.

\noindent\emph{Case 1:} $\rho >1/2$

By Lemma~\ref{lem:lower-bounds-of-resistance}, each $\abs{F_i}\geq 2$. Because the sets $F_1,\dots , F_m$ are pairwise disjoint,
\[
n\geq \sum_{i=1}^m \abs{F_i}\geq 2m.
\]

\noindent\emph{Case 2:} $0<\rho \leq 1/2$ and $m\geq 3$

Suppose, for a contradiction, that $n\leq m(\floor{1/\rho}+1)$. By Lemma~\ref{lem:lower-bounds-of-resistance}, each $\abs{F_i}\geq \floor{1/\rho}+1$, so
\begin{align*}
n\geq \sum_{i=1}^m \abs{F_i}
\geq m\left(\floor{\frac{1}{\rho}}+1\right).
\end{align*}
Therefore, $n=m(\floor{1/\rho}+1)$, and the sets $F_1,\dots,F_m$ partition $V(T)$, each with size $\floor{1/\rho}+1$. We will construct a graph whose
vertices are the blocks of this partition and show that
it is connected and that every vertex has degree $1$.
Such a graph must be isomorphic to $K_2$, which would
force $m=2$, contradicting our assumption that $m\geq 3$.

Let $\mathcal F$ be a graph with vertex set $\{F_1,\dots,F_m\}$ and an edge between $F_j$ and $F_\ell$ for $j\neq \ell$ if and only if there are vertices $x\in F_j$ and $y\in F_\ell$ that are adjacent in $T$.

Let $F_j$ and $F_\ell$ be two vertices of $\mathcal F$ with $j\neq \ell$. Take $x\in F_j$ and $y\in F_\ell$. 
Since $T$ is connected, there is a path between $x$ and $y$ in $T$. Since $F_1,\ldots,F_m$ form a partition of $V(T)$, each vertex of $T$ belongs to exactly one $F_i$. Thus, recording the sets containing the vertices
of this path, with consecutive repetitions removed, yields an $F_j, F_\ell$-walk in $\mathcal F$. Since every walk between two vertices contains a path between them  \cite[Lemma~1.2.5]{west2001}, there is an $F_j, F_\ell$-path in $\mathcal F$. Hence, $\mathcal F$ is connected.

Since $\abs{F_i}=\floor{1/\rho}+1$ for every $i$, Lemma~\ref{lem:lower-bounds-of-resistance} gives $\abs{\delta_T(F_i)}=1$; therefore, each vertex in $\mathcal F$ has degree $1$. However, every connected graph in which each vertex has degree $1$ is isomorphic to $K_2$, contradicting $\abs{\mathcal F}=m\geq 3$. Therefore,
\[
n\geq m\left(\floor{\frac{1}{\rho}}+1\right)+1.
\]

\noindent\emph{Case 3:} $0<\rho \leq 1/2$ and $m=2$

Lemma~\ref{lem:lower-bounds-of-resistance} gives $\abs{F_i}\geq \floor{1/\rho}+1$ for each $i$. As $F_1$ and $F_2$ are disjoint, we have
\[
n\geq 2\left(\floor{\frac{1}{\rho}}+1\right).
\]

\end{proof}
\begin{lemma}\label{lem:upper-bounds}
Let $T$ be a tree with $n=\abs{T}\geq 2$ and $m=h_{\rho}(T)$.
\begin{enumerate}
\item If $\rho>1/2$, then
\[
m\leq \floor{\frac{n}{2}}.
\]

\item Suppose that $0<\rho\leq 1/2$. Then
\[
m\leq
\begin{cases}
    1, & 2\leq n<2(\floor{1/\rho}+1),\\
    2, & 2(\floor{1/\rho}+1)\leq n\leq 3(\floor{1/\rho}+1),\\
    \floor{\dfrac{n-1}{\floor{1/\rho}+1}},
    & n\geq 3(\floor{1/\rho}+1)+1.
\end{cases}
\]
\end{enumerate}
\end{lemma}

\begin{proof}
If $\rho>1/2$, then Proposition~\ref{prop:min-order} gives the bound.

Now assume that $0<\rho\leq 1/2$. If
$n<2(\floor{1/\rho}+1)$, then the contrapositive of
Proposition~\ref{prop:min-order} gives $m<2$. Since $m\geq 1$, we have
$m=1$.

If
$2(\floor{1/\rho}+1)\leq n\leq 3(\floor{1/\rho}+1)$, then
Proposition~\ref{prop:min-order} gives $m\leq 2$.

Suppose that $n\geq 3(\floor{1/\rho}+1)+1$. If $m\geq 3$, then
Proposition~\ref{prop:min-order} gives
\[
m\leq \floor{\frac{n-1}{\floor{1/\rho}+1}}.
\]
If $m\leq 2$, the bound still holds.
\end{proof}

\subsection{Extremal trees}\label{extremal-trees}

\noindent\emph{Case 1:}  $\rho>1/2$ and $n\geq 2$.

Let $T=P_n$ with vertices $v_1,\dots,v_n$ in order. Since a degree-two
vertex has threshold $\ceil{2\rho}=2$, each of the sets
\[
\{v_1,v_2\},\ \{v_3,v_4\},\ \ldots
\]
is a resistant set. There are $\floor{n/2}$ pairwise disjoint such resistant sets, so
Corollary~\ref{cor:min-max} gives
\[
h_\rho(P_n)\geq\floor{\frac n2}.
\]
Together with the upper bound from Lemma~\ref{lem:upper-bounds}, this shows that $h_\rho(P_n)=\floor{n/2}$.
\\

\noindent\emph{Case 2:} $0<\rho\leq 1/2$  and $2\leq n < 2(\floor{1/\rho}+1)$.

The proof of Lemma~\ref{lem:upper-bounds} shows that $h_\rho(T)=1$ for every tree $T$ of order $n$ in this range.
\\

\noindent\emph{Case 3:} $0<\rho\leq 1/2$ and $2(\floor{1/\rho}+1)\leq n\leq 3(\floor{1/\rho}+1)$. 

Take a path with endpoints $x$ and $y$ and $n-2(\floor{1/\rho}+1)$ internal vertices. Attach $\floor{1/\rho}$ leaves to each of $x$ and $y$, and call this graph $T$. Then $\abs{T}=n$, and $T$ is a tree. The degree of each of $x$ and $y$ is $\floor{1/\rho}+1$, so
\[
\tau_{\rho}(x) = \ceil{\rho\floor{1/\rho}+\rho}=2.
\]
Thus, $x$ together with all of its leaves is a resistant set, and likewise for $y$. By Corollary~\ref{cor:min-max}, we get $h_{\rho}(T)\geq 2$. Together with the upper bound from Lemma~\ref{lem:upper-bounds}, this shows that $h_{\rho}(T)=2$. This construction is shown in Figure~\ref{fig:extremal-trees}(\subref{fig:extremal-case3}).
\\

\noindent\emph{Case 4:} $0<\rho\leq 1/2$ and $n\geq 3(\floor{1/\rho}+1)+1$.

Set 
\[
m=\floor{\frac{n-1}{\floor{1/\rho}+1}}.
\]
Let $c$ be a vertex, and attach vertices $x_1,\dots,x_m$ to $c$. Attach $\floor{1/\rho}$ leaves to each $x_i$. Take the edge $cx_1$ and subdivide it $n-(1+m(\floor{1/\rho}+1))$ times. Denote the resulting graph by $T$. It is clear that $T$ is a tree of order $n$. Each $x_i$, together with its leaf neighbors, forms a resistant set. These sets are pairwise disjoint; hence, $h_{\rho}(T)\geq m$. Together with the upper bound from Lemma~\ref{lem:upper-bounds}, this shows that $h_{\rho}(T)=m$. This construction is shown in Figure~\ref{fig:extremal-trees}(\subref{fig:extremal-case4}).

\begin{figure}[htbp]
    \centering

    \begin{subfigure}{\textwidth}
        \centering
        \begin{tikzpicture}[extremal tree]
  \coordinate (x)  at (-3.2,0);
  \coordinate (u1) at (-2.0,0);
  \coordinate (v1) at ( 2.0,0);
  \coordinate (y)  at ( 3.2,0);

  \draw[construction edge]
    (x)--(u1)--(-0.40,0);

  \draw[construction edge]
    (0.40,0)--(v1)--(y);

  \node[font=\large,fill=white,inner sep=2pt]
    at (0,0) {$\cdots$};

  \foreach \angle in {112.5,157.5,202.5,247.5}{
    \draw[construction edge]
      (x)--++(\angle:1.05)
      node[tree vertex] {};
  }

  \foreach \angle in {218,225,232}{
    \path (x) ++(\angle:1.05)
      node[circle,fill=black,inner sep=0pt,
           minimum size=0.9pt] {};
  }

  \foreach \angle in {-67.5,-22.5,22.5,67.5}{
    \draw[construction edge]
      (y)--++(\angle:1.05)
      node[tree vertex] {};
  }

  \foreach \angle in {-52,-45,-38}{
    \path (y) ++(\angle:1.05)
      node[circle,fill=black,inner sep=0pt,
           minimum size=0.9pt] {};
  }

  \foreach \p in {x,u1,v1,y}{
    \node[tree vertex] at (\p) {};
  }

  \node[tree label,below right=3pt] at (x) {$x$};
  \node[tree label,below left=3pt] at (y) {$y$};
\end{tikzpicture}
        \caption{The construction in Case 3.}
        \label{fig:extremal-case3}
    \end{subfigure}

    \par\medskip

    \begin{subfigure}{\textwidth}
        \centering
        \begin{tikzpicture}[extremal tree]
  \coordinate (center) at (0,0);
  \coordinate (s1) at (1.15,0);
  \coordinate (s2) at (2.30,0);
  \coordinate (x1) at (3.45,0);

  \draw[construction edge]
    (center)--(s1)--(s2)--(x1);

  \coordinate (x2) at (-2.4,1.4);
  \coordinate (xm) at (-2.4,-1.4);

  \draw[construction edge] (center)--(x2);
  \draw[construction edge] (center)--(xm);

  \foreach \angle in {-67.5,-22.5,22.5,67.5}{
    \draw[construction edge]
      (x1)--++(\angle:1.05)
      node[tree vertex] {};
  }

  \foreach \angle in {-52,-45,-38}{
    \path (x1) ++(\angle:1.05)
      node[circle,fill=black,inner sep=0pt,
           minimum size=0.9pt] {};
  }

  \foreach \angle in {22.5,67.5,112.5,157.5}{
    \draw[construction edge]
      (x2)--++(\angle:1.05)
      node[tree vertex] {};
  }

  \foreach \angle in {38,45,52}{
    \path (x2) ++(\angle:1.05)
      node[circle,fill=black,inner sep=0pt,
           minimum size=0.9pt] {};
  }

  \foreach \angle in {202.5,247.5,292.5,337.5}{
    \draw[construction edge]
      (xm)--++(\angle:1.05)
      node[tree vertex] {};
  }

  \foreach \angle in {308,315,322}{
    \path (xm) ++(\angle:1.05)
      node[circle,fill=black,inner sep=0pt,
           minimum size=0.9pt] {};
  }

  \foreach \p in {center,s1,s2,x1,x2,xm}{
    \node[tree vertex] at (\p) {};
  }

  \node[tree label,below right=3pt] at (center) {$c$};
  \node[tree label,below left =3pt] at (x1) {$x_1$};
  \node[tree label,below left=5pt] at (x2) {$x_2$};
  \node[tree label,above left=5pt] at (xm) {$x_m$};

  \node[font=\large] at (-1.75,0) {$\vdots$};

  \node[font=\large,fill=white,inner sep=1pt]
    at (1.72,0) {$\cdots$};

\end{tikzpicture}
        \caption{The construction in Case 4.}
        \label{fig:extremal-case4}
    \end{subfigure}

    \caption{Each of the vertices $x,y,x_1,\dots,x_m$ has $\floor{1/\rho}$ attached leaves.}
    \label{fig:extremal-trees}
\end{figure}
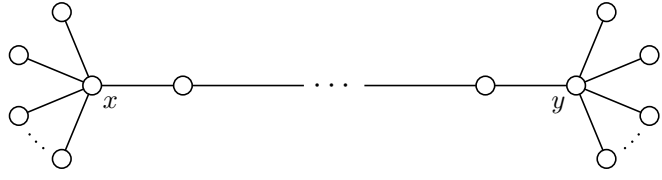
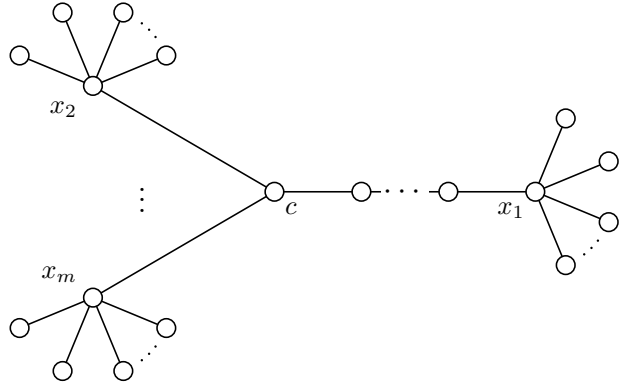

\section{Edge addition and $h_\rho$}

For nonadjacent vertices $x, y$ of a graph $G$, let $G +xy$ denote the graph obtained by adding the edge $xy$.

\medskip

\begin{lemma}\label{lem:process-comparison}
Let $G=H+xy$.
\begin{enumerate}
    \item For every $S\subseteq V(H)$,
    \[
    S_t^H\subseteq (S\cup\{x\})_t^G
    \]
    for all $t\ge 0$. By symmetry,
    \[
    S_t^H\subseteq (S\cup\{y\})_t^G
    \]
    for all $t\ge 0$.

    \item Let $C\subseteq V(G)$, and suppose, without loss of generality,
    that $x$ becomes active no later than $y$ in the activation process
    of $G$ starting from $C$. Then
    \[
    C_t^G\subseteq (C\cup\{y\})_t^H
    \]
    for all $t\ge 0$.
\end{enumerate}
\end{lemma}

\begin{proof}
For (1), we prove the first inclusion by induction on $t$. If $t=0$, then
\[
S_0^H=S\subseteq S\cup\{x\}=(S\cup\{x\})_0^G.
\]

Suppose that the claim holds for $t-1$, that is,
\[
S_{t-1}^H\subseteq (S\cup\{x\})_{t-1}^G.
\]
Let $z\in S_t^H$. If $z\in S_{t-1}^H$, then by the induction hypothesis,
\[
z\in S_{t-1}^H
\subseteq (S\cup\{x\})_{t-1}^G
\subseteq (S\cup\{x\})_t^G.
\]

Now suppose that $z\in S_t^H\setminus S_{t-1}^H$. Then
\[
|N_H(z)\cap S_{t-1}^H|
\ge
\left\lceil \rho d_H(z)\right\rceil.
\]

If $z=x$, then
\[
x\in (S\cup\{x\})_0^G
\subseteq (S\cup\{x\})_t^G.
\]

If $z\notin\{x,y\}$, then
\[
N_G(z)=N_H(z)
\quad\text{and}\quad
d_G(z)=d_H(z).
\]
Hence, by the induction hypothesis,
\[
\begin{aligned}
|N_G(z)\cap (S\cup\{x\})_{t-1}^G|
&\ge |N_H(z)\cap S_{t-1}^H|\\
&\ge \left\lceil \rho d_H(z)\right\rceil\\
&= \left\lceil \rho d_G(z)\right\rceil.
\end{aligned}
\]
Thus
\[
z\in (S\cup\{x\})_t^G.
\]

Finally, suppose that $z=y$. Since
\[
d_G(y)=d_H(y)+1,
\]
we have
\[
\left\lceil \rho d_G(y)\right\rceil
=
\left\lceil \rho(d_H(y)+1)\right\rceil
\le
\left\lceil \rho d_H(y)\right\rceil+1.
\]
Moreover, $x\in (S\cup\{x\})_{t-1}^G$ and $xy\in E(G)$. Since
$x\notin N_H(y)$, the induction hypothesis yields
\[
\begin{aligned}
|N_G(y)\cap (S\cup\{x\})_{t-1}^G|
&\ge |N_H(y)\cap S_{t-1}^H|+1\\
&\ge \left\lceil \rho d_H(y)\right\rceil+1\\
&\ge \left\lceil \rho d_G(y)\right\rceil.
\end{aligned}
\]
Therefore,
\[
y\in (S\cup\{x\})_t^G.
\]

Thus
\[
S_t^H\subseteq (S\cup\{x\})_t^G
\]
for every $t\ge 0$. The second inclusion follows by symmetry.

For (2), we again proceed by induction on $t$. If $t=0$, then
\[
C_0^G=C\subseteq C\cup\{y\}=(C\cup\{y\})_0^H.
\]

Assume that
\[
C_{t-1}^G\subseteq (C\cup\{y\})_{t-1}^H.
\]
Let $z\in C_t^G$. If $z\in C_{t-1}^G$, then
\[
z\in C_{t-1}^G
\subseteq (C\cup\{y\})_{t-1}^H
\subseteq (C\cup\{y\})_t^H.
\]
Thus, suppose that
\[
z\in C_t^G\setminus C_{t-1}^G.
\]

If $z=y$, then
\[
y\in C\cup\{y\}\subseteq (C\cup\{y\})_t^H.
\]

If $z\notin\{x,y\}$, then $N_G(z)=N_H(z)$ and
$d_G(z)=d_H(z)$. Hence, by the induction hypothesis,
\[
\begin{aligned}
|N_H(z)\cap (C\cup\{y\})_{t-1}^H|
&\ge |N_G(z)\cap C_{t-1}^G|\\
&\ge \lceil\rho d_G(z)\rceil\\
&=\lceil\rho d_H(z)\rceil.
\end{aligned}
\]
Thus, $z\in(C\cup\{y\})_t^H$.

Finally, suppose that $z=x$. Since
\[
x\in C_t^G\setminus C_{t-1}^G,
\]
$x$ becomes active at step $t$. By our assumption that $x$ becomes active
no later than $y$, we have
\[
y\notin C_{t-1}^G.
\]
Hence,
\[
y\notin N_G(x)\cap C_{t-1}^G.
\]
Since $y\notin C_{t-1}^G$ and
$N_G(x)=N_H(x)\cup\{y\}$, we have
\[
N_G(x)\cap C_{t-1}^G
=
N_H(x)\cap C_{t-1}^G
\subseteq
N_H(x)\cap(C\cup\{y\})_{t-1}^H,
\]
where the inclusion follows from the induction hypothesis.
Thus,
\[
\begin{aligned}
|N_H(x)\cap (C\cup\{y\})_{t-1}^H|
&\ge |N_G(x)\cap C_{t-1}^G|\\
&\ge \lceil \rho d_G(x)\rceil\\
&=\lceil \rho(d_H(x)+1)\rceil\\
&\ge \lceil \rho d_H(x)\rceil.
\end{aligned}
\]
Hence,
\[
x\in (C\cup\{y\})_t^H.
\]

Therefore,
\[
C_t^G\subseteq (C\cup\{y\})_t^H
\]
for every $t\ge 0$.
\end{proof}

\begin{theorem}\label{thm:added-edge}
Let $\rho\in[0,1]$ and let $G=H+xy$. Then
\[
|h_\rho(G)-h_\rho(H)|\le 1.
\]
\end{theorem}

\begin{proof}
Let $S$ be a minimum $\rho$-contagious set of $H$.
Then $S_t^H=V(H)$ for some $t$. By
Lemma~\ref{lem:process-comparison} (1),
\[
S_t^H\subseteq (S\cup\{x\})_t^G.
\]
Hence $S\cup\{x\}$ is a $\rho$-contagious set of $G$, and so
\[
h_\rho(G)\le h_\rho(H)+1.
\]

Conversely, let $C$ be a minimum $\rho$-contagious set of $G$.
Without loss of generality, assume that $x$ becomes active no
later than $y$. Since $C_t^G=V(G)$ for some $t$, by
Lemma~\ref{lem:process-comparison} (2),
\[
C_t^G\subseteq (C\cup\{y\})_t^H.
\]
Thus $C\cup\{y\}$ is a $\rho$-contagious set of $H$, and hence
\[
h_\rho(H)\le h_\rho(G)+1.
\]
Therefore,
\[
|h_\rho(G)-h_\rho(H)|\le1.
\]
\end{proof}

\begin{lemma}\label{lem:edge-preserves-resistance}
Let $\rho\in[0,1]$, let $H$ be a graph, let $F$ be a $\rho$-resistant set in $H$, and let
$x,y$ be nonadjacent vertices of $H$. If
\[
\{x,y\}\subseteq F
\qquad\text{or}\qquad
\{x,y\}\cap F=\emptyset,
\]
then $F$ is a $\rho$-resistant set in $G=H+xy$.
\end{lemma}

\begin{proof}
Under either assumption, the added edge does not have exactly one
endpoint in $F$. Consequently, for every $v\in F$,
\[
\abs{N_G(v)\setminus F}
=
\abs{N_H(v)\setminus F}.
\]
Since $F$ is a resistant set in $H$,
\[
\abs{N_G(v)\setminus F}
=
\abs{N_H(v)\setminus F}
<
\ceil{\rho d_H(v)}
\leq
\ceil{\rho d_G(v)}.
\]
Therefore, $F$ is a $\rho$-resistant set in $G$.
\end{proof}

\begin{lemma}\label{lem:leaf-edge-preserves-resistance}
Let $\rho\in(1/2,1]$, let $F$ be a $\rho$-resistant set in $H$, and
let $x\in F$ be a leaf of $H$ whose neighbor belongs to $F$.
If $y\in V(H)\setminus F$, then $F$ remains a
$\rho$-resistant set in $G=H+xy$.
\end{lemma}

\begin{proof}
The degrees and outside-neighbor counts of the vertices in
$F\setminus\{x\}$ are unchanged. At $x$, we have $ d_G(x)=2$ and 
$
\abs{N_G(x)\setminus F}=1$.
Since $\rho>1/2$,
\[
\abs{N_G(x)\setminus F}
=
1
<
\ceil{2\rho}
=
\ceil{\rho d_G(x)}.
\]
Therefore, $F$ is a $\rho$-resistant set in $G$.
\end{proof}
\begin{theorem}\label{thm:one-added-edge-tree}
Let $T$ be a tree of order $n\geq3$, and let $\rho\in[0,1]$. There exist two nonadjacent vertices $x,y\in V(T)$ such that for $G= T+xy$ we have 
\[
0\leq h_{\rho}(G)-h_\rho(T)\leq1.
\]
\end{theorem}

\begin{proof}
If $\rho=0$ then every vertex has threshold zero, so $h_0(T)=h_0(G)=0$, and the result follows. Thus, let $\rho \in (0,1]$ and set $m= h_{\rho}(T)$.

By Corollary~\ref{cor:min-max}, there exist $m$ pairwise disjoint
connected resistant sets
\[
F_1,\dots,F_m\]
in $T$. If $m=1$, choose any two nonadjacent vertices $x,y\in V(T)$ and let $G=T+xy$. Since all thresholds in $G$ are positive,
\[
h_{\rho}(G)\geq 1=m.
\]

Suppose now that $m\geq2$. Since the sets are pairwise disjoint, every
$F_i$ is proper. Thus, by Lemma~\ref{lem:lower-bounds-of-resistance}, we have
$\abs{F_i}\geq2$ for every $i$. We consider the following cases:

\noindent\emph{Case 1:}
$\abs{F_i}\geq3$
for some $i$. 

Since $T[F_i]$ is a tree of order at least three, we can choose two nonadjacent vertices $x,y \in F_i$. Let $G= T+xy$. The added edge has both endpoints in $F_{i}$ and neither endpoint in any $F_j$ with $j\neq i$. Thus, by Lemma~\ref{lem:edge-preserves-resistance}, each of $F_1,\dots,F_m$ remains a resistant set in $G$. We then have by Lemma~\ref{resistant graph prop}, 
\[
h_{\rho}(G) \geq m.
\]

\noindent\emph{Case 2:}
$\abs{F_i}=2$
for every $i$. 

By Lemma~\ref{lem:lower-bounds-of-resistance}, this case can occur only when $\rho>1/2$. 

Choose two nonadjacent leaves $x,y \in V(T)$ and let $G= T+xy$. Fix $j\in \{1,\dots, m\}$. Since $T[F_j]$ is connected and has order two, its vertices are adjacent. Thus, $F_j$ cannot contain both $x$ and $y$.

If $F_j \cap \{x,y\}=\emptyset$, then Lemma~\ref{lem:edge-preserves-resistance} implies that $F_j$ remains a resistant set in $G$.

Otherwise, $F_j$ contains exactly one of $x$ and $y$. By symmetry, assume that $x\in F_j$. Since $x$ is a leaf and $T[F_j]$ is connected, we have $F_j =\{x,u\}$, where $u$ is the neighbor of $x$ in $T$. Lemma~\ref{lem:leaf-edge-preserves-resistance} then implies that $F_j$ remains a resistant set in $G$. 

Thus, all $F_1,\ldots,F_m$ remain $\rho$-resistant sets
in $G$, and hence $h_\rho(G)\ge m$ by
Lemma~\ref{resistant graph prop}.

We have shown that an edge can always be added to $T$ to obtain a graph $G$ satisfying 
\[
h_{\rho}(G)\geq h_{\rho}(T).
\]

Combining this inequality with the bound in Theorem~\ref{thm:added-edge}, we obtain 
\[
0\leq h_{\rho}(G)- h_{\rho}(T)\leq 1.
\]
\end{proof}

\medskip

\section{Spectrum of $h_{\rho}$ on trees of order $n$}

We now turn from determining the maximum value $M_\rho(n)$ to identifying all values of $h_\rho$ attained by trees of order $n$. We first examine the effect of reattaching a leaf to another vertex.

\begin{theorem}\label{thm:gap}
Let $\rho\in[0,1]$, and let $T$ be an $n$-vertex tree with $n\ge3$.
Let $v$ be a leaf of $T$ with unique neighbor $u$.  For a vertex
$w\in V(T)\setminus\{u,v\}$, let $T' = T-uv+wv$.  
If $S$ is a minimum $\rho$-contagious set of $T$ containing no leaves, then $S' = S \cup \{w\}$ is a $\rho$-contagious set of $T'$.  
\end{theorem}
\begin{proof}
If $\rho=0$, the result is immediate. Thus, suppose that $\rho>0$.
We prove by induction on $t$ that
\[
S_t^T\subseteq (S')_t^{T'}
\]
for every $t\ge0$.

For $t=0$, this is immediate since
\[
S_0^T=S\subseteq S'=(S')_0^{T'}.
\]

Suppose that
\[
S_t^T\subseteq (S')_t^{T'}.
\]
We show that
\[
S_{t+1}^T\subseteq (S')_{t+1}^{T'}.
\]
Let $x\in S_{t+1}^T$. If $x\in S_t^T$, then by the induction
hypothesis,
\[
x\in (S')_t^{T'}\subseteq (S')_{t+1}^{T'}.
\]
Thus, suppose that
\[
x\in S_{t+1}^T\setminus S_t^T.
\]
We consider the following cases.

\medskip
\noindent
(1) $x=w$. Since $w\in S'$, we have
\[
w\in (S')_0^{T'}\subseteq (S')_{t+1}^{T'}.
\]

\medskip
\noindent
(2) $x=u$. Since $S$ contains no leaves, $v\notin S$.
Moreover, since $v$ is a leaf of $T$ with unique neighbor $u$,
$v$ cannot be activated before $u$. Since $u\notin S_t^T$, it follows
that $v\notin S_t^T.$
Hence,
\[
N_T(u)\cap S_t^T=N_{T'}(u)\cap S_t^T.
\]
Using the induction hypothesis, we obtain
\[
\begin{aligned}
|N_{T'}(u)\cap (S')_t^{T'}|
&\ge |N_{T'}(u)\cap S_t^T|\\
&=|N_T(u)\cap S_t^T|\\
&\ge \lceil \rho d_T(u)\rceil\\
&\ge \lceil \rho d_{T'}(u)\rceil.
\end{aligned}
\]
Thus,
\[
u\in (S')_{t+1}^{T'}.
\]

\medskip
\noindent
(3) $x=v$. Since $w\in S'$ and $v$ is adjacent to $w$ in $T'$,
\[
v\in (S')_1^{T'}\subseteq (S')_{t+1}^{T'}.
\]

\medskip
\noindent
(4) $x\notin\{u,v,w\}$. Since
\[
N_T(x)=N_{T'}(x)
\quad\text{and}\quad
d_T(x)=d_{T'}(x),
\]
the induction hypothesis gives
\[
\begin{aligned}
|N_{T'}(x)\cap (S')_t^{T'}|
&\ge |N_T(x)\cap S_t^T|\\
&\ge \lceil \rho d_T(x)\rceil\\
&=\lceil \rho d_{T'}(x)\rceil.
\end{aligned}
\]
Therefore,
\[
S_{t+1}^T\subseteq (S')_{t+1}^{T'}.
\]
By induction,
\[
S_t^T\subseteq (S')_t^{T'}
\]
for every $t\ge0$. Since $S$ is a $\rho$-contagious set of $T$,
we have $S_t^T=V(T)$ for some $t$. Hence,
\[
(S')_t^{T'}=V(T'),
\]
and therefore $S'$ is a $\rho$-contagious set of $T'$.
\end{proof}

\begin{definition}
Let $T$ be a tree, let $v$ be a leaf of $T$ with unique neighbor $u$,
and let $w\in V(T)\setminus\{u,v\}$. The operation
\[
T\longmapsto T-uv+vw
\]
is called a \emph{leaf exchange} at $v$.
\end{definition}
We may identify the vertex sets of $T$ and $T'$ with $V(K_n)$.
Then both $T$ and $T'$ can be viewed as spanning trees of $K_n$. Since $K_n$ is $2$-connected, its leaf-exchange spanning tree graph
is connected by \cite[Theorem~1]{leaf1996}. Hence, we obtain the
following corollary.

\begin{corollary}\label{cor:exchange}
Let $T$ and $T'$ be two $n$-vertex trees on the same vertex set,
where $n\ge3$. Then there exists a sequence
\[
T=T_0,T_1,\ldots,T_k=T'
\]
of $n$-vertex trees such that for each $i=0,\ldots,k-1$, $T_{i+1}$ is obtained from
$T_i$ by a leaf exchange.
\end{corollary}

\begin{corollary}\label{cor:leaf-exchange}
Let $T$ and $T'$ be trees of order $n\ge3$,
where $T'$ is obtained from $T$ by a leaf exchange.
Then
\[
|h_\rho(T')-h_\rho(T)|\le1.
\]
\end{corollary}

\begin{proof}
By Theorem~\ref{thm:gap},
\[
h_\rho(T')\le h_\rho(T)+1.
\]
Since $T$ can also be obtained from $T'$ by a
leaf exchange, applying the same theorem in the
reverse direction gives
\[
h_\rho(T)\le h_\rho(T')+1.
\]
The result follows.
\end{proof}

For $\rho\in(0,1]$ and $n\ge2$, define
\[
\mathcal H_\rho(n)
=
\{h_\rho(T): T \text{ is an $n$-vertex tree}\}.
\]

\begin{theorem}\label{thm:spectrum}
For every $\rho\in(0,1]$ and $n\ge2$,
\[
\mathcal H_\rho(n)
=
\{1,2,\ldots,M_\rho(n)\}.
\]
\end{theorem}

\begin{proof}
For $n = 2$, the statement is trivial. Assume that $n \ge 3$. 
Let $T_0$ be the star with $n$ vertices. Also, let $T'$ be a tree with $h_\rho(T') = M_\rho(n) = M$. By Corollary~\ref{cor:exchange} there exists a sequence \[
T_0,T_1,\ldots,T_k=T'
\]
of $n$-vertex trees such that for each $i=0,\ldots,k-1$, $T_{i+1}$ is obtained from
$T_i$ by a leaf exchange.

For each $i$, let
\[
h_i=h_\rho(T_i).
\]
Then
\[
h_0=1
\quad\text{and}\quad
h_k=M.
\]
Moreover, by Corollary~\ref{cor:leaf-exchange},
\[
|h_{i+1}-h_i|\le1
\]
for every $i=0,\ldots,k-1$.

Suppose that some integer in $\{1,2,\ldots,M\}$ does not occur among
$h_0,h_1,\ldots,h_k$, and let $q$ be such an integer. Since
$h_0=1$ and $h_k=M$, we have
\[
1<q<M.
\]
Let $j$ be the smallest index such that
\[
h_j>q.
\]
Then, since $q$ does not occur among $h_0,\ldots,h_k$,
\[
h_{j-1}<q<h_j.
\]
Since the $h_i$ are integers,
\[
|h_j-h_{j-1}|\ge2,
\]
contradicting
\[
|h_j-h_{j-1}|\le1.
\]
Therefore every integer in $\{1,2,\ldots,M\}$ occurs.
\end{proof}

\section*{Acknowledgements}
Jeong-Ok Choi was supported by the GIST Research Project
grant funded by the Gwangju Institute of Science and Technology (GIST)
in 2026. Collin McDevitt was supported by the Global Intern Program
at the Gwangju Institute of Science and Technology.

\section*{ Declaration of AI-assisted technologies}
During the preparation of this work, the authors used ChatGPT
(OpenAI, GPT-5.6 Sol) to assist with language editing, to provide
suggestions on the organization and flow of the manuscript, and to
check whether results obtained independently in this work had already
appeared in the literature in related forms, particularly where
differences in terminology or language might have made them difficult
to identify through conventional searches.   These tools were also used to assist with portions of selected proofs, particularly the proof of Proposition~\ref{prop:min-order}(2). After using this tool, the
authors reviewed and edited the content as needed and take full
responsibility for the content of the published article.

\section*{Data Availability}
No data was used for the research described in the article.

\bibliographystyle{abbrv}
\bibliography{bib_cir.bib}

\end{document}